\documentclass[12pt]{article}
\usepackage{amssymb,amsmath,latexsym}
\usepackage{amsbsy}
\usepackage{amsfonts}
\usepackage{amsopn}
\usepackage{amsthm}
\usepackage{amsthm}
\usepackage{amsmath,verbatim,amsrefs}

\begin{document}

\begin{doublespace}

\newtheorem{theorem}{Theorem}
\newtheorem{prop}[theorem]{Proposition}
\newtheorem{lemma}[theorem]{Lemma}
\newtheorem{defi}[theorem]{Definition}
\newtheorem{corol}[theorem]{Corollary}
\newtheorem{conjec}[theorem]{Conjecture}
\newtheorem{claim}[theorem]{Claim}
\newtheorem{remark}[theorem]{Remark}

\renewcommand{\qedsymbol}{$\blacksquare$}

\newtheorem{thm}{Theorem}[section]
\newtheorem{defn}{Definition}[section]
\newtheorem{corollary}[thm]{Corollary}
\newtheorem{example}[thm]{Example}
\numberwithin{equation}{section}

\newcommand{\D}{{\rm d}}
\def\ee{\varepsilon}
\def\qed{{\hfill $\Box$ \bigskip}}
\def\MM{{\cal M}}
\def\BB{{\cal B}}
\def\LL{{\cal L}}
\def\FF{{\cal F}}
\def\EE{{\cal E}}
\def\QQ{{\cal Q}}

\def\R{{\mathbb R}}
\def\L{{\bf L}}
\def\E{{\mathbb E}}
\def\F{{\bf F}}
\def\P{{\mathbb P}}
\def\N{{\mathbb N}}
\def\eps{\varepsilon}
\def\wh{\widehat}
\def\pf{\noindent{\bf Proof.} }

\title{\Large \bf Optimal control with absolutely continuous strategies for spectrally negative L\'evy processes\thanks{
R.L. gratefully acknowledges support from the AXA Research Fund. J-L.P. acknowledges financial support from CONACyT grant number 000000000129326.
}}
\author{Andreas E. Kyprianou\thanks{Department of Mathematical
Sciences, University of Bath, Claverton Down, Bath, BA2 7AY, U.K.
E-mail: a.kyprianou@bath.ac.uk}, \, Ronnie Loeffen\thanks{	Weierstrass Institute for Applied Analysis and Stochastics,
Mohrenstrasse 39,
10117 Berlin
Germany. E-mail: loeffen@wias-berlin.de} \, and Jos\'e-Luis P\'erez\thanks{Department of Mathematical
Sciences, University of Bath, Claverton Down, Bath, BA2 7AY, U.K.
E-mail: jlapg20@bath.ac.uk}}
\date{\today}

\maketitle

\begin{abstract}
In the last few years there has been renewed interest in the classical
control problem of de Finetti \cite{definetti} for the case that
underlying source of randomness is a spectrally negative L\'evy process.
In particular a significant step forward is made in
\cite{loeffen1} where it is shown that a natural and very general condition
on the underlying L\'evy process which allows one to proceed with the
analysis of the associated Hamilton-Jacobi-Bellman equation is that its L\'evy
measure is absolutely continuous, having completely monotone density.

In this paper we consider
 de Finetti's control problem
 but now with the restriction that control strategies
are absolutely continuous with respect to Lebesgue measure.
This problem has been considered  by Asmussen and Taksar
\cite{asmussentaksar}, Jeanblanc and Shiryaev \cite{jeanblanc}
and Boguslavskaya \cite{bog} in the diffusive case and Gerber
and Shiu \citelist{\cite{gerbershiucompound}}
for the case of a Cram\'er-Lundberg process with exponentially
distributed jumps.  We show the robustness of the condition
that the underlying L\'evy measure has a completely monotone
density and establish an explicit optimal strategy for this
case that envelopes the aforementioned existing results.
The explicit optimal strategy in question is the so-called
refraction strategy.
\end{abstract}



\section{Introduction and main result}\label{introduction}
Recently there  there has been a growing body of literature
 which explores the interaction of
classical models of ruin and fluctuation theory of L\'evy processes; see for example
\cites{palmdiv, DK2006, huzak, huzak2, KKM2004,  kyppalm, kypriverosong,
loeffen1, loeffen2, loeffen3, loeffenrenaud, biffiskyp}.
Of particular note in this respect is the application of the theory of scale functions
for spectrally negative L\'evy processes. This article adds to the aforementioned list by addressing a modification of de Finetti's classical dividend problem through the theory of scale functions. Before turning to our main results, let us first attend to the basic definitions of the mathematical objects that we are predominantly interested in.

Recall that a spectrally negative L\'evy process is a stochastic process issued from the
origin which has c\`adl\`ag
paths and stationary and independent
increments
such that there are no positive discontinuities. To avoid degenerate cases in the forthcoming discussion, we shall additionally exclude from this definition the case of monotone paths. This means that we are not interested in the case of a deterministic increasing linear drift or the negative of a subordinator.

Henceforth we assume that $X=\{X_t: t\ge 0\}$ is a spectrally negative
L\'evy process under $\mathbb{P}$ with L\'evy triplet given by $(\gamma, \sigma, \nu)$,
where $\gamma \in \R$, $\sigma\ge 0$ and $\nu$ is a measure concentrated on
$(0,\infty)$ satisfying
$$
\int_{(0,\infty)}(1\wedge z^2)\nu({\rm d}z)<\infty.
$$
The Laplace exponent of $X$ is given by
\begin{equation*}
\psi(\lambda)= \log \mathbb{E} \left[ \mathrm{e}^{\lambda X_1} \right] = \gamma \lambda +\frac12\sigma^2\lambda^2 -
\int_{(0,\infty)} \left( 1-\mathrm{e}^{-\lambda z} -\lambda z\mathbf{1}_{\{0<z\leq1\}} \right)\nu(\mathrm{d}z),
\end{equation*}
which is well defined for $\lambda\geq0$. Here $\mathbb{E}$ denotes expectation with respect to $\mathbb{P}$.
The reader will note that, for convenience, we have arranged the
representation of the Laplace exponent in such a way that the
support of the L\'evy measure is positive even though the process
experiences only negative jumps.
As a strong Markov process we shall endow $X$ with probabilities
$\{\mathbb{P}_x : x\in\mathbb{R}\}$ such that under $\mathbb{P}_x$
we have $X_0 = x$ with probability one. Note that $\mathbb{P}_0= \mathbb{P}$.

It is well-known that $X$ has paths of bounded variation if and only if $\sigma=0$ and $\int_0^1 z\nu(\mathrm{d}z)<\infty$. In this case $X$ can be written as
\begin{equation}
X_t=ct-S_t, \,\, t\geq 0,
\label{BVSNLP}
\end{equation}
where $c=\gamma+\int_0^1z\nu(\mathrm{d}z)$ and $\{S_t:t\geq0\}$ is a driftless subordinator. Note that  necessarily $c>0$, since we have ruled out the case that $X$ has monotone paths.
Moreover, when $\nu(0,\infty)<\infty$, then $X$ is known in the actuarial mathematics literature as the classical Cram\'er-Lundberg risk process. This processs is often used to  model the surplus wealth of an insurance company.

The classical theory of ruin concerns itself  with the path of the stochastic risk process until the moment that it first passes below the level zero; the event corresponding to ruin. An offshoot of the classical ruin problem was introduced by de Finetti \cite{definetti}.
His intention was
to make the study of ruin more
realistic by introducing the possibility that dividends are paid out
to shareholders up to the moment of ruin.
Further, the payment of
dividends should be made in such a way as to optimize the expected
net present value of the total dividends paid to the shareholders
from time zero until ruin. Mathematically speaking, de Finetti's
dividend problem amounts to solving a control problem which we state
in the next paragraph. Although de Finetti's  dividend problem has its origin in insurance mathematics, there are several papers \citelist{\cite{bayraktar} \cite{belhaj} \cite{radner}} that have considered this problem into the context of corporate finance.

 Let $\pi=\{L^\pi_t: t\geq 0\}$ be
a dividend strategy, meaning that it is a left-continuous non-negative
non-decreasing process adapted to the (completed and right
continuous) filtration $\mathbb{F}:=\{\mathcal{F}_t : t\geq 0\}$ of $X$. The
quantity $L^\pi_t$ thus represents the cumulative dividends paid out
up to time $t$ by the insurance company whose risk process is
modelled by $X$. An additional constraint on  $\pi$ is that  $L^\pi_{t+} - L^\pi_t\leq \max\{U^\pi_t, 0\}$ for $t\leq \sigma^\pi$ (i.e. lump sum dividend payments are always smaller than the available reserves at any time before ruin). The $\pi$-controlled L\'evy process is thus $U^\pi=\{U^\pi_t :
t\geq 0\}$ where $U^\pi_t = X_t - L^\pi_t$. Write $\sigma^\pi
=\inf\{t>0: U^\pi_t < 0 \}$ for the time at which ruin occurs when
the dividend payments are taken into account. 
Suppose that  $\Pi$ denotes some family of admissible strategies, which we shall elaborate on later. Then the expected net present
value  of the dividend policy $\pi\in\Pi$
with discounting at rate $q>0$ and initial capital $x\geq 0$ is given by
$$
v_\pi(x) = \mathbb{E}_x\left[\int_{[0,{\sigma^\pi}]} \mathrm{e}^{-qt}
\mathrm{d}L_t^\pi\right],
$$
where $\mathbb{E}_x$ denotes expectation with respect to
$\mathbb{P}_x$ and $q>0$ is a fixed rate. {\it De Finetti's dividend
problem} consists of characterising the optimal value function,
\begin{equation}
\label{controlproblem}
v_*(x):= \sup_{\pi\in\Pi} v_\pi(x),
\end{equation}
and, further, if it exists, establish a strategy $\pi^*$ such that
$v_*(x) = v_{\pi^*}(x)$.

In the case that $\Pi$ consists of all  strategies as described at the beginning of the previous paragraph  there are now extensive results in the literature, most of which have appeared in the last few years.
Initially this problem was considered  by Gerber \cite{gerber69} who proved that,
for the Cram\'{e}r-Lundberg model with exponentially distributed
jumps, the optimal value function is the result of a {\it reflection
strategy}. That is to say, a strategy of the form $L^a_t =
a\vee\overline{X}_t - a$ for some  optimal $a\geq 0$ where $\overline{X}_t
:=\sup_{s\leq t}X_s$. In that case the controlled process $U^{a}_t =
X_t - L^a_t$ is a spectrally negative L\'evy process reflected at
the barrier $a$. However, a sequence of innovative works \cites{palmdiv, azcue, loeffen1, loeffen2, kypriverosong, loeffenrenaud}
have pushed this conclusion much further into the considerably more general setting where $X$ is a spectrally negative L\'evy process. Of particular note amongst these references is the paper of Loeffen \cite{loeffen1} in which the optimality of the reflection strategy is shown to depend in a very subtle way on the shape of the so-called {\it scale functions} associated to the underlying L\'evy process. Indeed Loeffen's new perspective on de Finetti's control problem lead
to very easily verifiable sufficient conditions for the reflection strategy to be optimal. Loeffen shows that it suffices for the L\'evy measure $\nu$ to be absolutely continuous with a completely monotone density, thereby allowing for a very large family of general spectrally negative L\'evy processes. Through largely technical adaptations of Loeffen's method, this sufficient condition was relaxed in \cites{kypriverosong, loeffenrenaud}. It is important to note that in general a barrier strategy  is not  always an optimal strategy; an explicit counter-example was  provided by Azcue and Muler \cite{azcue}.

In this article we are interested in addressing an adaptation of de Finetti's dividend problem by considering a smaller class of admissible strategies. Specifically, we are interested in the case that, in addition to the assumption that strategies are non-decreasing and $\mathcal{F}$-adapted, $\Pi$ only admits absolutely continuous strategies $\pi=\{L_t^\pi:t\geq0\}$ such that
\begin{equation*}
L_t^\pi=\int_0^t \ell^\pi(s)\mathrm{d}s,
\end{equation*}
and  for $t\geq0$, $\ell^\pi(t)$ satisfies
\begin{equation*}
0\leq \ell^\pi(t) \leq \delta,
\end{equation*}
where $\delta>0$ is a ceiling rate.  Moreover, we make the assumption that
\begin{equation}
\label{assumptionH}
\delta<\gamma+\int_0^1z\nu(\mathrm{d}z) \quad \text{if $X$ has paths of bounded variation.} \tag{\textbf{H}}
\end{equation}
Note that for a reflection strategy, when $X$ has paths of unbounded variation the corresponding dividend process is supported by increase times that are  singular with respect to Lebesgue measure  and when $X$ has paths of bounded variation, the dividend process is supported by increase times which are absolutely continuous with respect to Lebesgue measure  with rate $c$ (cf. Section 6.1 of \cite{kypbook}). Recalling the decomposition (\ref{BVSNLP}) and assumption \eqref{assumptionH}, we see that the reflection strategy is therefore not included into the smaller class of admissible controls and we are left with a truly different control problem. A particular motivation for studying this kind of modification is that, if a reflection strategy is applied, then the company will get ruined in finite time with probability one, which is seen as an undesired consequence. By restricting the set of admissible strategies in the way described above and provided $\mathbb{E}[X_1]>\delta$, we make 
 sure that there is a strictly positive probability that ruin will never occur no matter which admissible dividend strategy is applied.


The reader familiar with optimal control problems of this kind, will recognize that the optimal strategy should be of bang-bang type, i.e. depending on the value of the controlled process, dividends should either be paid out at the maximum rate $\delta$ or at the minimum rate $0$. A particularly simple bang-bang strategy is the one that we refer to here as a {\it refraction strategy}, which, in words, is the strategy where dividends are paid out at the maximum rate when the controlled process is above a certain level $b\geq0$ and at the minimum rate when below $b$. Mathematically, a refraction strategy at $b$ is the strategy which corresponds to the controlled process taking the form of the unique strong solution to the following stochastic differential equation,
\begin{equation}
\label{SDE}
\mathrm{d}U^b_t= \mathrm{d}X_t - \delta\mathbf{1}_{\{U^b_t>b\}}{\rm d}t,\,\, t\geq 0.
\end{equation}
In the case of absolutely continuous control strategies for $X$, it has been shown by Asmussen and Taksar
\cite{asmussentaksar}, Jeanblanc and Shiryaev \cite{jeanblanc}
and Boguslavskaya \cite{bog} in the diffusive case and by Gerber
and Shiu \cite{gerbershiucompound}
for the case of a Cram\'er-Lundberg process with exponentially
distributed jumps that a refraction strategy, where $b\geq 0$ is optimally chosen, is optimal. This particular control problem is also discussed in the review papers of Avanzi \cite{avanzi} and Albrecher and Thonhauser \cite{albrecherthonhauseroverview} and in the book of Schmidli \cite{schmidlibook}.


In the spirit of earlier work for the more general class of admissible strategies, the point of view we shall take here is to deal with a general spectrally negative L\'evy process and give sufficient conditions under which a refraction strategy of the form (\ref{SDE}) is optimal. Note that when $X$ is a general spectrally negative L\'evy process, the strong existence and uniqueness of solutions to (\ref{SDE})  under (H), so called {\it refracted L\'evy processes}, were established in Kyprianou and Loeffen \cite{kyploeffen}. Our main result is the following.

\begin{theorem}\label{main}
Suppose the L\'evy measure  has a completely monotone density. Then an optimal strategy for the control problem is formed by a refraction strategy.
\end{theorem}

The above theorem offers the same sufficient condition on the L\'evy measure as Loeffen \cite{loeffen1} for the larger,  general class of admissible strategies.
Although, as alluded to above, weaker assumptions have been established in that case, the technical details of our method appears to not to  allow us to follow suit. To illustrate the difference between the two cases, we give in Remark \ref{counterexample} a specific example of $X$ and $q$ for which no refraction strategy can be optimal in the restricted case for a certain choice of the ceiling rate $\delta$,  whereas a reflection strategy is  optimal within the general class of admissible dividend strategies.
We also remark that in fact our method allows us to give a more quantitative result than Theorem \ref{main} in the sense that we are able to characterise the threshold $b^*$ associated with the optimal refraction strategy. As some more notation is needed to do this,  it is given at the end of the paper in Corollary \ref{corol}.

\bigskip

We close this section with a brief summary of the remainder of the
paper. In the next section we show the role played by scale functions in giving a workable identity for the expected net present value of a refraction strategy. We also use this identity to describe an appropriate candidate for the threshold associated with the optimal refraction strategy. Then in the final section we put together a series of technical lemmas which allow us to verify the optimality of the identified threshold strategy. The assumption that $\nu$ has a completely monotone density will repeatedly play a very significant role in the aforementioned lemmas.

\section{Scale functions and refraction strategies}

As alluded to above, a key element of the forthcoming analysis relies on the theory of so-called scale functions. We therefore devote some time in this section reminding the reader of some fundamental properties of scale functions as well as their relevance to refraction strategies.

For
each $q\geq0$ the so called $q$-scale function of $X$, $W^{(q)}:
\R\to [0, \infty),$ is the unique function such that $W^{(q)}(x)=0$
for $x<0$ and on $[0, \infty)$ is a strictly increasing and
continuous function whose Laplace transform is given by
\begin{equation}
\int^{\infty}_0\mathrm{e}^{-\theta x}W^{(q)}(x)dx=\frac1{\psi(\theta)-q},
\quad \theta>\Phi(q).
\label{scalefn}
\end{equation}
Here
\[
 \Phi(q) = \sup\{\lambda \geq 0: \psi(\lambda) = q\}
     \]
and is well defined and finite for all $q\geq 0$ as a consequence of the well known fact that $\psi$ is a strictly convex function satisfying $\psi(0) = 0$ and $\psi(\infty) = \infty$.

Shape and smoothness properties of the scale functions $W^{(q)}$ will be of particular interest to us in the forthcoming analysis.
In the discussion below we shall consider the behaviour of $W^{(q)}$ at $0$, $\infty$ as well as describing qualitative features of its shape on $(0,\infty)$.
We start with  some standard facts concerning the behaviour of the scale function in the neighbourhood of the origin. Recall that we have defined the constant
\[
  c= \gamma+\int_0^1 z \, \nu(\mathrm{d}z)
\]
in the case that $X$ has bounded variation paths.

The following result is well known and can easily be deduced from (\ref{scalefn}). See for example Chapter 8 of \cite{kypbook}.
\begin{lemma}\label{prop:Wq0}
As $x\downarrow 0$, the value of the scale function $W^{(q)}(x)$ and its right derivative are
determined for every $q\geq 0$ as follows
\begin{equation}\label{initialvalues}
\begin{split}
W^{(q)}(0+)= &
\begin{cases}
1/c & \text{when $\sigma=0$ and $\int_0^1 z \, \nu(\mathrm{d}z)<\infty$},  \\
0 & \text{otherwise},
\end{cases}\\
W^{(q)\prime}(0+)=&
\begin{cases}
2/\sigma^2 & \text{when $\sigma>0$,} \\
(\nu(0,\infty)+q)/c^2 & \text{when $\sigma=0$ and $\nu(0,\infty)<\infty$,} \\
\infty & \text{otherwise.}
\end{cases}
\end{split}
\end{equation}
\end{lemma}

 In general it is know that one may always write for $q\geq 0$
\begin{equation}
 W^{(q)}(x) = \mathrm{e}^{\Phi(q)x}W_{\Phi(q)}(x),
 \label{exp-growth}
\end{equation}
where $W_{\Phi(q)}$ plays the role of a $0$-scale function of an auxilliary spectrally negative L\'evy process with Laplace exponent given by $\psi_{\Phi(q)}(\lambda) = \psi(\lambda + \Phi(q)) - q$. Note the fact that $\psi_{\Phi(q)}$ is the Laplace exponent follows by an exponential tilting argument, see for example Chapter 8 of Kyprianou \cite{kypbook}. In the same reference one also sees that $\lim_{x\uparrow\infty}W_{\Phi(q)}(x)<1/\psi'(\Phi(q))<\infty$, which suggests that, when $q>0$, the function $W^{(q)}(x)$ behaves like the exponential function $\mathrm{e}^{\Phi(q)x}$ for large $x$. It is therefore natural to ask whether $W^{(q)}(x)$ is convex for large values of $x$. This very question was addressed in Loeffen \cites{loeffen1, loeffen2}. In these papers it was found that, due to quite a deep connection between scale functions and potential measures of subordinators, a natural assumption which allows one to address the issue of convexity, and, in fact, say a lot more
 ,
  is that the L\'evy measure $\nu$ is absolutely continuous with completely monotone density.  In the next lemma we collect a number of consequences of this assumption, lifted from the aforementioned two papers. We need first some more notation. Recalling that $W^{(q)}$ is continuously differentiable on $(0,\infty)$ as soon as  $\nu$ has no atoms (see for example the discussion in \cite{chankypsavov}), a key quantity in the lemma is the constant
\[
 a^* = \sup\{a\geq 0: W^{(q)\prime}(a) \leq W^{(q)\prime}(x) \text{ for all } x\geq 0\}.
\]
Note that we understand $W^{(q)\prime}(0)$ to mean $W^{(q)\prime}(0+)$ above and necessarily $a^*<\infty$ since, by (\ref{exp-growth}), we have that $\lim_{x\uparrow\infty} W^{(q)\prime}(x) =\infty$.

\begin{lemma}
 \label{compmon}
Suppose the L\'evy measure has a completely monotone density and $q>0$. Then
the $q$-scale function can be written as
\begin{equation*}
W^{(q)}(x)=\Phi'(q) \mathrm{e}^{\Phi(q)x} -f(x), \quad x>0,
\end{equation*}
where $f$ is a non-negative, completely monotone function.
Moreover, $W^{(q)\prime}$ is strictly log-convex  (and hence convex) on $(0,\infty)$. Since $W^{(q)\prime}(\infty) = \infty$, $a^*$ is thus the unique point at which $W^{(q)\prime}$ attains its minimum so that  $W^{(q)\prime}$ is strictly decreasing on $(0,a^*)$ and strictly increasing on $(a^*,\infty)$.
\end{lemma}

\bigskip

Let us now progress to a description of the role played by scale functions in connection with the value of a refraction strategy. In addition to the scale function $W^{(q)}$ associated to the spectrally negative L\'evy process $X$, we shall also define for each $q\geq0$ the scale functions $\mathbb W^{(q)}$ which are associated to the linearly perturbed spectrally negative L\'evy process $Y = \{Y_t : t\geq 0\}$ where  $ Y_t = X_t-\delta t$ for $t\geq0$. Note that because of Assumption \eqref{assumptionH} the aforementioned process does not have monotone paths.
 Further we denote by $\varphi(q)$ the right inverse of the Laplace exponent of $Y$, i.e.
 \[
 \varphi(q)=\inf\{\lambda\geq0:\psi(\lambda)-\delta\lambda=q\}.
 \]
The value function of the refraction strategy at level $b$, henceforth denoted by $v_b$, can now be written explicitly in terms of  $W^{(q)}$, $\mathbb{W}^{(q)}$ and $\varphi(q)$.
Indeed it was shown in Equation (10.25) of \cite{kyploeffen} that
\begin{equation}
\label{valuethres}
v_b(x)= -\delta\int_0^{x-b}\mathbb W^{(q)}(y)\mathrm{d}y + \frac{W^{(q)}(x)+\delta\int_b^x \mathbb W^{(q)}(x-y)W^{(q)\prime}(y)\mathrm{d}y}{h(b)}, \quad x\geq0,
\end{equation}
where $h(b)$ is given by
\begin{equation}
\label{hb}
h(b)=\varphi(q)\mathrm{e}^{\varphi(q)b}\int_b^\infty \mathrm{e}^{-\varphi(q)y}W^{(q)\prime}(y)\mathrm{d}y
= \varphi(q)\int_0^\infty \mathrm{e}^{-\varphi(q)u}W^{(q)\prime}(u+b)\mathrm{d}u.
\end{equation}
 Note that
\begin{equation}
\label{valuebelowthreshold}
v_b(x)=\frac{W^{(q)}(x)}{h(b)} \quad \text{for $x\leq b$.}
\end{equation}

\bigskip

We need also to have a candidate optimal threshold, say $b^*$, in combination with the expression for $v_b$ if we are to check for optimality. To this end define
 $b^*$ as the largest argument at which $h$ attains its minimum. That is to say,
\begin{equation*}
b^*=\sup\{b\geq0: h(b)\leq h(x) \ \text{for all $x\geq0$}\}.
\end{equation*}

Under the same conditions as   Theorem \ref{main}, we are able to say some more about $b^*$.

\begin{lemma}\label{b*} Suppose that  that $\nu$ has a completely monotone density, then  $b^*\in[0,a^*)$ and it is the unique point at which $h$ attains its minimum. Moreover, $b^* >0$ if and only if one of the following three cases hold,
\begin{description}
\item[(i)] $\sigma>0$ and $\varphi(q)<2\delta/\sigma^2$,
\item[(ii)] $\sigma = 0$, $\nu(0,\infty)<\infty$ and $\varphi(q)<\delta(\nu(0,\infty) + q)/c(c-\delta)$ or
\item[(iii)] $\sigma = 0$ and $\nu(0,\infty)=\infty$.
\end{description}
\end{lemma}

\begin{proof}[\textbf{Proof}]
We begin by showing that $b^*<\infty$. 
Note that
\begin{equation*}
\begin{split}
h(b)
=& (\varphi(q))^2 \int_0^\infty \mathrm{e}^{-\varphi(q) y} \left( W^{(q)}(y+b)-W^{(q)}(b) \right) \mathrm{d}y \\
=& (\varphi(q))^2\mathrm{e}^{\Phi(q)b}\int_0^\infty\mathrm{e}^{-\varphi(q)y} \left( \mathrm{e}^{\Phi(q)y}W_{\Phi(q)}(y+b)-W_{\Phi(q)}(b) \right) \mathrm{d}y\\
\geq & (\varphi(q))^2\mathrm{e}^{\Phi(q)b}W_{\Phi(q)}(b)\int_0^\infty\mathrm{e}^{-\varphi(q) y} \left( \mathrm{e}^{\Phi(q)y}-1\right) \mathrm{d}y\\
= & W^{(q)}(b)\frac{\varphi(q)\Phi(q)}{\varphi(q)-\Phi(q)},
\end{split}
\end{equation*}
where we have  used a change of variables and an integration by parts for the first equality.  Since $W^{(q)}(\infty)=\infty$ it follows that   $\lim_{b\rightarrow\infty}h(b)=\infty$ which in turn implies the finiteness of $b^*$.

From \eqref{hb}, we see that $h$ is continuously differentiable and that
\begin{equation}
\label{relationhscale}
h'(b)=\varphi(q) \left( h(b) - W^{(q)\prime}(b) \right).
\end{equation}
It follows immediately that $h'(b) > (<)0$ if and only if $h(b) > (<)W^{(q)\prime} (b)$. Thanks to Lemma \ref{compmon}, we know that $W^{(q)\prime}$ is a strictly convex function satisfying $W^{(q)\prime}(\infty) = \infty$ and since $h(\infty) = \infty$, it follows that, there is a unique $b^*\in[0,\infty)$ for which the minimum of $h$ is attained and $h(b)<W^{(q)\prime}(b)$ for $b<b^*$  and $h(b) >W^{(q)\prime}(b)$ when $b>b^*$.  Moreover, when  $b^*>0$, we have that $h(b^*) =W^{(q)\prime}(b^*)$.

Let us now show that $b^*<a^*$. Suppose for contradiction that $b^*>a^*$. In that case, since $W^{(q)\prime\prime}(b)>0$ for all $b\geq b^*$ and $h'(b^*)=0$, it follows that there exists a sufficiently small $\epsilon>0$ such that $W^{(q)\prime}(b)>h(b)$ for all $b\in(b^*,b^*+\epsilon)$. However this last statement contradicts the earlier conclusion that $W^{(q)\prime}(b)<h(b)$ for all $b>b^*$. Now suppose, also for contradiction, that $b^* = a^*$. Considering the second equality in (\ref{hb}), since  $W^{(q)\prime}(u+a^*) > W^{(q)\prime}(a^*)$ for all $u> 0$, it is straightforward to show that
$
h(a^*) > W^{(q)\prime}(a^*)
$
which again contradicts our earlier conclusion that $h(b^*) = W^{(q)}(b^*)$.


Finally, given that $b^*$ characterises the single crossing point of the function $h$ over the function $W^{(q)\prime}$,  we have that $b^*>0$ if and only if $h(0)<W^{(q)\prime}(0+)$. Note from (\ref{hb}) that
\begin{equation}
\label{h0}
h(0)=\varphi(q) \left( \frac{\varphi(q)}{\psi(\varphi(q))-q} - W^{(q)}(0) \right) =  \varphi(q)  \left( \frac1{\delta} - W^{(q)}(0) \right)
\end{equation}
where we have  used the fact that for $q>0$, that by integration by parts in (\ref{scalefn}),
\begin{equation}
\int_{[0,\infty)} e^{-\theta x} W^{(q)}({\rm d}x) = \frac{\theta}{\psi(\theta) -q} ,\,\, \theta >\Phi(q).
\label{IBP}
\end{equation}
and that $\varphi(q)>\Phi(q)$. The three cases that are equivalent to  $b^*>0$ now follow directly from the right hand side of  (\ref{h0})  compared against the expression given for $W^{(q)\prime}(0+)$ in Lemma \ref{prop:Wq0}.
\end{proof}

\section{Verification}\label{verify}

For the remainder of the paper, we will focus on verifying  the optimality of the refraction strategy at threshold level $b^*$ under the condition that $\nu$ has a completely monotone density.

Given the spectrally negative L\'evy process $X$, we call a function $f$ (defined on at least the positive half line) sufficiently smooth if $f$ is continuously differentiable on $(0,\infty)$ when $X$ has paths of bounded variation and is twice continuously differentiable on $(0,\infty)$ when $X$ has paths of unbounded variation. We let $\Gamma$ be the operator acting on sufficiently smooth functions $f$, defined by
\begin{equation*}
\begin{split}
\Gamma f(x)= \gamma f'(x)+\frac{\sigma^2}{2}f''(x) +\int_{(0,\infty)}[f(x-z)-f(x)+f'(x)z\mathbf{1}_{\{0<z\leq1\}}]\nu(\mathrm{d}z).
\end{split}
\end{equation*}

The following lemma constitutes standard technology as far as optimal control is concerned. For this reason its proof, which requires only a technical modification of Lemma 1 in \cites{loeffen2}, is deferred to the appendix.

\begin{lemma}
\label{verificationlemma}
Suppose $\hat{\pi}$ is an admissible dividend strategy such that $v_{\hat{\pi}}$ is sufficiently smooth on $(0,\infty)$, right-continuous at zero and for all $x>0$
\begin{equation}
\label{HJB-inequality}
\sup_{0\leq r\leq\delta}\Gamma v_{\hat{\pi}}(x)-q v_{\hat{\pi}}(x)-rv'_{\hat{\pi}}(x)+r\leq 0.
\end{equation}
Then $v_{\hat{\pi}}(x)=v_*(x)$ for all $x\geq0$ and hence $\hat{\pi}$ is an optimal strategy.
\end{lemma}

As we wish to work with this lemma for the case that $v_{\hat\pi}  = v_{b*}$, we show next that $v_{b^*}$ is sufficiently smooth.

\begin{lemma}\label{suffsmooth} Under the assumption of Theorem \ref{main},
the value function $v_{b^*}$ is sufficiently smooth.
\end{lemma}
\begin{proof}[{\bf Proof}]
Recall from Lemma \ref{compmon} that when $\nu$ has a completely monotone density it follows that both $W^{(q)}$ and $\mathbb{W}^{(q)}$ are infinitely differentiable.

Now suppose that $b^* = 0$. Then from (\ref{valuethres}) it follows that
 \begin{equation}
v_0(x)= -\delta \left( \int_0^{x}\mathbb W^{(q)}(y)\mathrm{d}y - \frac1{\varphi(q)} \mathbb W^{(q)}(x) \right), \quad x\geq0,
\label{v0}
\end{equation}
which is clearly sufficiently smooth.

Next suppose that $b^*>0$. By differentiating \eqref{valuethres}, we get
\begin{equation}
\label{dervaluethresh}
v_{b^*}'(x)=-{\delta}\mathbb W^{(q)}(x-b^*)+\frac{(1+\delta \mathbb W^{(q)}(0)) W^{(q)\prime}(x)+\delta\int_{b^*}^x \mathbb W^{(q)\prime}(x-y)W^{(q)\prime}(y)\mathrm{d}y }{W^{(q)\prime}(b^*)}.
\end{equation}
Using an integration by parts in \eqref{dervaluethresh} leads to
\begin{equation}
\label{simpleform}
 v_{b^*}'(x)=\frac{  W^{(q)\prime}(x)+\delta\int_{b^*}^x \mathbb W^{(q)}(x-y)W^{(q)\prime\prime}(y)\mathrm{d}y }{W^{(q)\prime}(b^*)},
\end{equation}
which is continuous in $x$.
Differentiating (\ref{simpleform}) leads us to
\begin{eqnarray}
\label{secondderiv}
 v_{b^*}''(x)&=&\frac{  W^{(q)\prime\prime}(x)+\delta\mathbb{W}^{(q)}(0)W^{(q)\prime\prime}(x) + \delta\int_{b^*}^x \mathbb W^{(q)\prime}(x-y)W^{(q)\prime\prime}(y)\mathrm{d}y }{W^{(q)\prime}(b^*)}.
\end{eqnarray}
The expression on the right hand side  is clearly continuous in $x$ when $X$ has paths of unbounded variation as $\mathbb{W}^{(q)}(0)  =0$.
\end{proof}

Inspired by the cases that $X$ is diffusive or a Cram\'er-Lundberg process with exponentially distributed jumps, for which a solution to
the control problem at hand is known, we move next to the following two lemmas which convert the Hamilton-Jacobi-Bellman inequality in Lemma \ref{verificationlemma} into a more user friendly sufficient condition.

\begin{lemma}
\label{tussenlemma} Under the assumption of Theorem \ref{main}
the value function $v_{b^*}$ satisfies \eqref{HJB-inequality} if and only if
\begin{equation}
\label{equiv_inequality2}
\begin{cases}
v_{b^*}'(x)\geq1 & \text{if $0<x\leq b^*$}, \\
v_{b^*}'(x)\leq1 & \text{if $x>b^*$}.
\end{cases}
\end{equation}
\end{lemma}
\begin{proof}[\textbf{Proof}]
We first establish the following two equalities:
\begin{equation}
\label{generators}
\begin{split}
(\Gamma-q) v_{b^*}(x)=0  & \quad \text{for $0<x\leq b^*$},\\
(\Gamma-q) v_{b^*}(x) -\delta v_{b^*}'(x)+\delta =0 & \quad  \text{for $x>b^*$}.
\end{split}
\end{equation}
Recalling  \eqref{valuebelowthreshold} and the fact  that $v_{b^*}$ is sufficiently smooth, the first part of \eqref{generators} is proved in Lemma 4 of \cite{palmdiv} (see also \cite{biffiskyp}). In a similar way, the second part follows after we show that $M=\{M_t,t\geq0\}$ given by
\begin{equation*}
M_t= \mathrm{e}^{-q(t\wedge\tau_{b^*}^-)} \left( v_{b^*}(Y_{t\wedge\tau_{b^*}^-})-\frac{\delta}{q} \right), t\geq0
\end{equation*}
is a $\mathbf{P}_x$-martingale for $x>{b^*}$; here $\tau_{b^*}^-$ stands for $\tau_{b^*}^-=\inf\{t>0:Y_t<b^*\}$ and $\mathbf{P}_x$ is the law of $Y$ when $Y_0 = x$. Indeed, the martingale property follows by the following two computations and the tower property of conditional expectation. First we have for $x>{b^*}$ by the strong Markov property,
\begin{equation*}
\begin{split}
\mathbf{E}_x & \left.\left[ \mathrm{e}^{-q\tau_{b^*}^-} \left( v_{b^*}(Y_{\tau_{b^*}^-})-\frac{\delta}{q}  \right) \right| \mathcal{F}_t \right] \\
= & \mathbf{E}_x   \left.\left[ \mathbf{1}_{\{t<\tau_{b^*}^-\}} \mathrm{e}^{-q\tau_{b^*}^-} \left( v_{b^*}(Y_{\tau_{b^*}^-})-\frac{\delta}{q}  \right) \right| \mathcal{F}_t \right] + \mathbf{E}_x   \left.\left[ \mathbf{1}_{\{t\geq\tau_{b^*}^-\}} \mathrm{e}^{-q\tau_{b^*}^-} \left( v_{b^*}(Y_{\tau_{b^*}^-})-\frac{\delta}{q}  \right) \right| \mathcal{F}_t \right] \\
= & \mathbf{1}_{\{t<\tau_{b^*}^-\}}\mathrm{e}^{-qt}\mathbf{E}_{Y_t} \left[ \mathrm{e}^{-q\tau_{b^*}^-} \left( v_{b^*}(Y_{\tau_{b^*}^-})-\frac{\delta}{q}  \right) \right]
   + \mathbf{1}_{\{t\geq\tau_{b^*}^-\}} M_t.
\end{split}
\end{equation*}
Here $\mathbf{E}_x$ denotes expectation with respect to $\mathbf{P}_x$.
Recall that $U^{b^*}$ is the refracted L\'evy process given by (\ref{SDE}) with threshold $b^*$. Let $\kappa_{b^*}^-=\inf\{t>0:U_t^{b^*}<b^*\}$, then
\begin{equation*}
\begin{split}
\mathbf{1}_{\{t<\tau_{b^*}^-\}} M_t = & \mathbf{1}_{\{t<\tau_{b^*}^-\}}\mathrm{e}^{-qt} \mathbb{E}_{Y_t} \left[ \delta\int_0^{\sigma^{b^*}} \mathrm{e}^{-qs}\mathbf{1}_{\{U^{b^*}_s\in(b^*,\infty)\}} \mathrm{d}s - \frac{\delta}q \right] \\
= & \mathbf{1}_{\{t<\tau_{b^*}^-\}}\mathrm{e}^{-qt} \mathbb{E}_{Y_t} \left[ \delta\int_{\kappa_{b^*}^-}^{\sigma^{b^*}} \mathrm{e}^{-qs}\mathbf{1}_{\{U^{b^*}_s\in(b^*,\infty)\}} \mathrm{d}s - \frac{\delta}q \mathrm{e}^{-\kappa_{b^*}^-} \right] \\
= & \mathbf{1}_{\{t<\tau_{b^*}^-\}}\mathrm{e}^{-qt}\mathbf{E}_{Y_t} \left[ \mathrm{e}^{-q\tau_{b^*}^-} \left( v_{b^*}(Y_{\tau_{b^*}^-})-\frac{\delta}{q}  \right) \right],
\end{split}
\end{equation*}
where we used in the last line that given $Y_0=U_0$, $\{Y_t,0\leq t\leq \tau_{b^*}^-\}$ is equal in law to $\{U^{b^*}_t,0\leq t\leq \kappa_{b^*}^-\}$.

We now continue with the proof of the lemma.
It is easily seen that
  condition \eqref{HJB-inequality} is equivalent to
\begin{equation}
\label{equiv_inequality}
\begin{cases}
(\Gamma-q) v_{b^*}(x) \leq 0 & \text{if $v'_{b^*}(x)\geq1$}, \\
(\Gamma-q) v_{b^*}(x) -\delta v'_{b^*}(x)+\delta \leq 0 & \text{if $v'_{b^*}(x)<1$}.
\end{cases}
\end{equation}

Suppose now \eqref{equiv_inequality2} holds.  If $v_{b^*}'(x)>1$, then \eqref{equiv_inequality2} implies $x\leq b^*$ and so by \eqref{generators} $(\Gamma-q) v_{b^*}(x)=0$. If $v_{b^*}'(x)<1$, then \eqref{equiv_inequality2} implies $x> b^*$ and so by \eqref{generators} $(\Gamma-q) v_{b^*}(x) -\delta v_{b^*}'(x)+\delta =0$. If $v_{b^*}'(x)=1$, then we have by \eqref{generators} $(\Gamma-q) v_{b^*}(x)=0$. Hence \eqref{equiv_inequality} holds.

Suppose now \eqref{equiv_inequality} holds. Let $0<x\leq b^*$ and suppose $v_{b^*}'(x)<1$. Then  \eqref{generators} and \eqref{equiv_inequality} implies $-\delta v_{b^*}'(x)+\delta\leq0$ which implies $v_{b^*}'(x)\geq 1$ which forms a contradiction. Hence we deduce  $v_{b^*}'(x)\geq 1$.
Now let $x>b^*$ and suppose $v_{b^*}'(x)>1$. Then  \eqref{generators} and \eqref{equiv_inequality} implies $\delta v_{b^*}'(x)-\delta\leq0$ which implies $v_{b^*}'(x)\leq 1$ which forms a contradiction. Hence we deduce  $v_{b^*}'(x)\leq 1$.
\end{proof}

The following lemma forms the  most difficult part of the proof of the main theorem. It is here that Lemma \ref{compmon} and thus the assumption of complete monotonicity on the density of the L\'evy measure is most crucially needed.
\begin{lemma}
\label{cruciallemma}
Suppose the L\'evy measure  has a completely monotone density. Then the function $v_{b^*}$ satisfies \eqref{equiv_inequality2}.
\end{lemma}
\begin{proof}[\textbf{Proof}]
Suppose first that $b^*=0$. In other words, from Lemma \ref{b*}, assume that either
\begin{itemize}
\item[(i)] $\sigma>0$ and $\varphi(q)\geq 2\delta/\sigma^2$, or
\item[(ii)] $\sigma = 0$, $\nu(0,\infty)<\infty$ and $\varphi(q)\geq \delta(q+\nu(0,\infty))/c(c-\delta)$.
\end{itemize}
 Then for $x>0$, we deduce from (\ref{v0}) that
\begin{equation*}
v_{b^*}'(x)= v_{0}'(x)=-\delta \left( \mathbb W^{(q)}(x)  - \frac1{\varphi(q)} \mathbb W^{(q)\prime}(x) \right).
\end{equation*}
By the decomposition of the scale function given in Lemma \ref{compmon}, $v_0'$ is completely monotone  and thus in particular decreasing on $(0,\infty)$.  
Hence if $b^*=0$, it is enough to show that $v_{0}'(0+)\leq1$ or equivalently
\begin{equation}
\label{B}
\frac{\delta\mathbb W^{(q)\prime}(0+)}{1+\delta \mathbb W^{(q)}(0)}\leq \varphi(q).
\end{equation}
Taking account of Lemma \ref{prop:Wq0} we see that this requirement is automatically satisfied in cases (i) and (ii).
Hence we have proved  \eqref{equiv_inequality2} if $b^*=0$.

\bigskip

Assume now $b^*>0$. Then for $x\leq b^*$, $v_{b^*}'(x)=W^{(q)\prime}(x)/W^{(q)\prime}(b^*)$. From this it follows that $v_{b^*}'(x)\leq 1$ since, by Lemma \ref{b*}, $b^*\leq a^*$ and, by Lemma \ref{compmon}, $W^{(q)\prime}$ is decreasing for $x\leq b^*$.


\bigskip

Suppose now $x>b^*$. Differentiating twice the first displayed equation in Section 8 of \cite{kyploeffen} gives us the identity
\begin{equation*}
\delta\int_0^x \mathbb W^{(q)\prime}(x-y)W^{(q)\prime}(y)\mathrm{d}y = (1- \delta W^{(q)}(0))\mathbb W^{(q)\prime}(x) - (1+\delta\mathbb W^{(q)}(0))W^{(q)\prime}(x).
\end{equation*}
Hence revisiting (\ref{dervaluethresh}) we obtain the expression
\begin{equation*}
v_{b^*}'(x)=-{\delta}\mathbb W^{(q)}(x-b^*)+\frac{(1- \delta W^{(q)}(0))\mathbb W^{(q)\prime}(x)-\delta\int_0^{b^*} \mathbb W^{(q)\prime}(x-y)W^{(q)\prime}(y)\mathrm{d}y }{W^{(q)\prime}(b^*)}.
\end{equation*}
Appealing to  Lemma \ref{compmon} and writing $\mathbb{W}^{(q)}(x)=\varphi'(q)\mathrm{e}^{\varphi(q)x}-f(x)$ where $f$ is completely monotone we get
\begin{equation*}
\begin{split}
v_{b^*}'(x)= & {\delta}  f(x-b^*)+\frac{-(1- \delta W^{(q)}(0))f'(x)+\delta\int_0^{b^*} f'(x-y)W^{(q)\prime}(y)\mathrm{d}y }{W^{(q)\prime}(b^*)}  -\delta\varphi'(q)\mathrm{e}^{\varphi(q)(x-b^*)} \\
&+\frac1{W^{(q)\prime}(b^*)} \Big( (1- \delta W^{(q)}(0))\varphi'(q)\varphi(q)\mathrm{e}^{\varphi(q)x}  -\delta\int_0^{b^*} \varphi'(q)\varphi(q)\mathrm{e}^{\varphi(q)(x-y)}W^{(q)\prime}(y)\mathrm{d}y  \Big).
\end{split}
\end{equation*}
Using (\ref{IBP}) and recalling that $\varphi(q)>\Phi(q)$ we also have that
\begin{equation*}
\begin{split}
\int_0^{b^*}  \mathrm{e}^{-\varphi(q)y }W^{(q)\prime}(y)\mathrm{d}y = & -\int_{b^*}^\infty  \mathrm{e}^{-\varphi(q)y }W^{(q)\prime}(y)\mathrm{d}y + \int_0^\infty  \mathrm{e}^{-\varphi(q)y }W^{(q)\prime}(y)\mathrm{d}y \\
= & -\int_{b^*}^\infty  \mathrm{e}^{-\varphi(q)y }W^{(q)\prime}(y)\mathrm{d}y +   \frac{1}{\delta}-W^{(q)}(0).
\end{split}
\end{equation*}
Hence this gives us
\begin{equation*}
\begin{split}
v_{b^*}'(x)
= & {\delta}  f(x-b^*)+\frac{-(1- \delta W^{(q)}(0))f'(x)+\delta\int_0^{b^*} f'(x-y)W^{(q)\prime}(y)\mathrm{d}y }{W^{(q)\prime}(b^*)}  \\
& -\delta\varphi'(q)\mathrm{e}^{\varphi(q)(x-b^*)}
+ \frac{\delta\mathrm{e}^{\varphi(q)x} \varphi'(q)\varphi(q)  \int_{b^*}^\infty \mathrm{e}^{-\varphi(q)y}W^{(q)\prime}(y)\mathrm{d}y   }{W^{(q)\prime}(b^*)}
\end{split}
\end{equation*}
and recalling that
\begin{equation*}
W^{(q)\prime}(b^*) = h(b^*)=\varphi(q)\mathrm{e}^{\varphi(q)b^*}\int_{b^*}^\infty\mathrm{e}^{-\varphi(q)y}W^{(q)\prime}(y)\mathrm{d}y,
\end{equation*}
we get the simpler expression
\begin{equation*}
\begin{split}
v_{b^*}'(x)
= & {\delta}  f(x-b^*)+\frac{-(1- \delta W^{(q)}(0))f'(x)+\delta\int_0^{b^*} f'(x-y)W^{(q)\prime}(y)\mathrm{d}y }{W^{(q)\prime}(b^*)}.
\end{split}
\end{equation*}
Since $f$ is completely monotone, by Bernstein's theorem, it can be written in the form $f(x)=\int_0^\infty \mathrm{e}^{-xt}\mu(\mathrm{d}t)$ for some measure $\mu$. Therefore using Tonelli's Theorem, we come to rest at the identity
\begin{equation*}
v_{b^*}'(x)
= \int_0^\infty \mathrm{e}^{-xt}  \left\{ \delta\mathrm{e}^{b^*t} + \frac{1- \delta W^{(q)}(0)}{W^{(q)\prime}(b^*)} t - \frac{\delta}{W^{(q)\prime}(b^*)}\int_0^{b^*} t\mathrm{e}^{yt}W^{(q)\prime}(y)\mathrm{d}y \right\} \mu(\mathrm{d}t).
\end{equation*}

Denote for $t>0$ by $g(t)$ the expression between curly brackets. We have
\[
g''(t)=\delta (b^*)^2\mathrm{e}^{b^*t}    - \frac{2\delta}{W^{(q)\prime}(b^*)}\int_0^{b^*}  y\mathrm{e}^{yt}W^{(q)\prime}(y)\mathrm{d}y
- \frac{\delta  t}{W^{(q)\prime}(b^*)}\int_0^{b^*}y^2\mathrm{e}^{yt}W^{(q)\prime}(y)\mathrm{d}y.
\]
Since $W^{(q)\prime}(y)\geq W^{(q)\prime}(b^*)$ for $y\in(0,b^*)$, we have using also an integration by parts,
\begin{equation*}
\begin{split}
\frac{\delta  t}{W^{(q)\prime}(b^*)}\int_0^{b^*}y^2\mathrm{e}^{yt}W^{(q)\prime}(y)\mathrm{d}y \geq & \delta  t\int_0^{b^*}y^2\mathrm{e}^{yt} \mathrm{d}y
= \delta y^2\mathrm{e}^{yt}|_0^{b^*} - 2\delta\int_0^{b^*} y\mathrm{e}^{yt}\mathrm{d}y
\end{split}
\end{equation*}
and hence using again that $W^{(q)\prime}(y)\geq W^{(q)\prime}(b^*)$ for $y\in(0,b^*)$,
\begin{equation*}
g''(t)\leq      - \frac{2\delta}{W^{(q)\prime}(b^*)}\int_0^{b^*}  y\mathrm{e}^{yt}W^{(q)\prime}(y)\mathrm{d}y
+ 2\delta\int_0^{b^*} y\mathrm{e}^{yt}\mathrm{d}y \leq0.
\end{equation*}

In conclusion, $g$ is a concave function and in particular there exists $0\leq p_1\leq\infty$ such that $g$ is  increasing on $(0,p_1)$ and decreasing on $(p_1,\infty)$. Since $g(0)=\delta$, it follows that there exists $0<p_2\leq\infty$ such that $g$ is  positive on $(0,p_2)$ and negative on $(p_2,\infty)$. Consequently $\mathrm{e}^{-(x-b^*)t}g(t)\geq \mathrm{e}^{-(x-b^*)p_2}g(t)$ for all $t>0$ and thus (noting that we are allowed to switch the derivative and the integral)
\begin{equation}
\label{concavity}
\begin{split}
v_{b^*}''(x)=   -\int_0^\infty \mathrm{e}^{-(x-b^*)t}\mathrm{e}^{-b^*t}tg(t)\mathrm{d}t
 \leq -\mathrm{e}^{-(x-b^*)p_2}\int_0^\infty \mathrm{e}^{-b^*t}tg(t)\mathrm{d}t = \mathrm{e}^{-(x-b^*)p_2} v_{b^*}''(b^*+).
\end{split}
\end{equation}
From (\ref{secondderiv})
we easily deduce that
\begin{equation*}
v_{b^*}''(b^*+) = \frac{ (1+\delta \mathbb W^{(q)}(0)) W^{(q)\prime\prime}(b^*)} {W^{(q)\prime}(b^*)}\leq0
\end{equation*}
where the inequality is a result of the fact that  $b^*<a^*$ and hence by Lemma \ref{compmon}, $W^{(q)\prime\prime}(b^*)\leq 0$. In combination with \eqref{concavity}, it follows that $v_{b^*}'$ is decreasing on $(b^*,\infty)$ and since  $v_{b^*}'(b^*)=1$, we deduce that $v_{b^*}'(x)\leq 1$ for $x>b^*$ as required.
\end{proof}

\bigskip

Finally we can put all the pieces together to establish our main result.

\begin{proof}[\textbf{Proof of Theorem \ref{main}}]
 Simply combine  Lemmas \ref{verificationlemma}, \ref{suffsmooth}, \ref{tussenlemma}  and \ref{cruciallemma}.
\end{proof}

\begin{remark}\label{counterexample}\rm
Here we give the example mentioned at the end of Section \ref{introduction}. Let $\nu(\mathrm{d}z)=10z\mathrm{e}^{-z}\mathrm{d}z$, $\gamma=20.67-\int_0^1z\nu(\mathrm{d}z)$ and $q=0.1$. That is, $X$ is a Cram\'er-Lundberg risk process with premium rate $20.67$, jump arrival rate $10$ and with a Gamma$(2,1)$ claim distribution. The scale function for $X$, via the method of partial fraction, is given by
\begin{equation*}
W^{(q)}(x)=\sum_{i=1}^3 D_i\mathrm{e}^{\theta_i x}, \quad x\geq0,
\end{equation*}
where $\{\theta_i: i=1,2,3\}$ are the (distinct) roots of
$\lambda\mapsto\psi(\theta) -q$
with $\theta_1>0$ and $\theta_2,\theta_3<0$ and where $\{D_i: i=1,2,3\}$ are given by $D_i=1/\psi'(\theta_i)$. Similarly, the scale function of $Y$ can be calculated. If one plots the second derivative of $W^{(q)}$, then one  sees that $W^{(q)\prime\prime}(x)>0$ on $(0,\infty)$ and therefore by Theorem 2 of Loeffen \cite{loeffen1}, the reflection strategy at the barrier $0$ is optimal for the control problem with general admissible dividend strategies. If we now consider the case with absolutely continuous dividend strategies and choose the upper bound $\delta$ equal to $20.59$, then one can check that the function $h$ is strictly increasing and thus $v_0(0)=W^{(q)}(0)/h(0)>W^{(q)}(0)/h(b)=v_b(0)$ for all $b>0$. This means that the only refraction strategy that can be optimal is the one with threshold level equal to zero. However, one can calculate that $v_0'(3.15)=1.0005>1$ and so $v_0$ does not satisfy \eqref{equiv_inequality2} and consequently by Lemma \ref{tussenlemma}, $
 v_0$ does not satisfy \eqref{HJB-inequality}. Since we are in the Cram\'er-Lundberg setting with a continuous claim distribution, we can deduce from Theorem 2.32 of Schmidli \cite{schmidlibook} that the optimal value function $v_*$ has to satisfy \eqref{HJB-inequality} and therefore $v_0\neq v_*$. We conclude that for this particular example we have the remarkable property  that the optimal strategy in the case with no extra restrictions on the controls, is to always pay out the maximum amount of dividends that is allowed, whereas in the restricted case with this particular value of $\delta$, it is not optimal to always pay out dividends at the maximum rate. Further, this example shows that  Theorem 2 of \cite{loeffen1} does not carry over to the control problem considered in this paper.
\end{remark}

On a final note, Lemma \ref{b*} can also be re-worded as  the following corollary, giving a characterisation of the optimal threshold.

\begin{corol}\label{corol}
The threshold for the optimal refraction strategy in Theorem \ref{main} is characterised as the unique point  $b^* \in [0,a^*)$ such that
\[
b^*=\sup\{b\geq0: h(b)\leq h(x) \ \text{for all $x\geq0$}\}
\]
where
\[
h(b)=\varphi(q)\mathrm{e}^{\varphi(q)b}\int_b^\infty \mathrm{e}^{-\varphi(q)y}W^{(q)\prime}(y)\mathrm{d}y.
\]
Moreover $b^*>0$ if and only if
\begin{description}
\item[(i)] $\sigma>0$ and $\varphi(q)<2\delta/\sigma^2$,
\item[(ii)] $\sigma = 0$, $\nu(0,\infty)<\infty$ and $\varphi(q)<\delta(\nu(0,\infty) + q)/c(c-\delta)$ or
\item[(iii)] $\sigma = 0$ and $\nu(0,\infty)=\infty$.
\end{description}
otherwise $b^*=0$.
\end{corol}

\section*{Appendix: proof of Lemma \ref{verificationlemma}}
By definition of $v_*$, it follows that $v_{\hat{\pi}}(x)\leq v_*(x)$ for all $x\geq0$. We write $w:=v_{\hat{\pi}}$ and show that $w(x)\geq v_\pi(x)$ for all $\pi\in\Pi$ for all $x\geq0$. First we suppose $x>0$. We define for $\pi\in\Pi$ the stopping time $\sigma^\pi_0$ by $\sigma_0^\pi=\inf\{t>0:U_t^\pi\leq0\}$
and denote by $\Pi_0$ the following set of admissible dividend strategies
\begin{equation*}
\Pi_0=\{\pi\in\Pi:  v_\pi(x)= \mathbb{E}_x\left[\int_0^{\sigma_0^\pi}\mathrm{e}^{-qt}\mathrm{d}L_t^\pi \right] \ \text{for all $x>0$}\}.
\end{equation*}
We claim that any $\pi\in\Pi$ can be approximated by dividend strategies from $\Pi_0$ in the sense that for all $\epsilon>0$ there exists $\pi_{\epsilon}\in\Pi_0$ such that $v_\pi(x)\leq v_{\pi_\epsilon}(x)+\epsilon$ and therefore it is enough to show that $w(x)\geq v_\pi(x)$ for all $\pi\in\Pi_0$. Indeed, we can take $\pi_\epsilon$ to be the strategy where you do not pay out any dividends until the stopping time $\kappa:=\inf\{t>0:L_t^\pi\geq\epsilon\}$, then from that time point $\kappa$ onwards follow the same strategy as $\pi$ until ruin occurs for the latter strategy at which point you stop paying out dividends. Note that $\pi_\epsilon\in\Pi_0$ because if $\sigma_0^{\pi_\epsilon}<\kappa$, then $\sigma_0^{\pi_\epsilon}=\sigma^{\pi_\epsilon}$ since until the first dividend payment is made, the process $U^{\pi_\epsilon}$ is equal to $X$ and for the spectrally negative L\'evy process $X$, the first entry time in $(-\infty,0]$ is equal almost surely to the first entry time in
  $(-\infty,0)$, provided $X_0>0$. Further if $\sigma_0^{\pi_\epsilon}\geq\kappa$ and $\kappa<\infty$, then $\sigma_0^{\pi_\epsilon}\geq \sigma^\pi$ and so by construction there are no dividends paid out in the time interval $(\sigma_0^{\pi_\epsilon},\sigma^{\pi_\epsilon})$.

We now assume without loss of generality that $\pi\in\Pi_0$. Let $(T_n)_{n\in\mathbb{N}}$ be the sequence of stopping times defined by $T_n=\inf\{t>0:{U}^\pi_t>n \ \text{or} \ {U}^\pi_t<\frac{1}{n}\}$.
Since ${U}^\pi$ is a semi-martingale and $w$ is sufficiently smooth, we can use the the change of variables/It\^o's formula (cf. \cite{protter}*{Theorem II.31 \& II.32})  on $\mathrm{e}^{-q(t\wedge T_n)}w({U}^\pi_{t\wedge T_n})$ to deduce under $\mathbb{P}_x$,
\begin{equation*}
\label{impulse_verif_1}
\begin{split}
\mathrm{e}^{-q(t\wedge T_n)}w({U}^\pi_{t\wedge T_n})-w(x)
= & -\int_{0+}^{t\wedge T_n}\mathrm{e}^{-qs} q w({U}^\pi_{s-}) \mathrm{d}s
+\int_{0+}^{t\wedge T_n}\mathrm{e}^{-qs}w'({U}^\pi_{s-}) \mathrm{d}  ( X_s - {L}^\pi_s  )  \\
 & + \sum_{0<s\leq t\wedge T_n}\mathrm{e}^{-qs}[\Delta w({U}^\pi_{s-}+\Delta X_s)-w'({U}^\pi_{s-})  \Delta X_s  ],
\end{split}
\end{equation*}
where we use the following notation: $\Delta {U}^\pi_s={U}^\pi_s-{U}^\pi_{s-}$, $\Delta w({U}^\pi_s)=w({U}^\pi_s)-w({U}^\pi_{s-})$.
Rewriting the above equation leads to
\begin{equation*}
\begin{split}
\mathrm{e}^{-q(t\wedge T_n)}w({U}^\pi_{t\wedge T_n}) & -w(x)
\\
= &  \int_{0+}^{t\wedge T_n}\mathrm{e}^{-qs}   (\Gamma-q)w({U}^\pi_{s-})   \mathrm{d}s
 -\int_{0+}^{t\wedge T_n}\mathrm{e}^{-qs}w'({U}^\pi_{s-})\mathrm{d}{L}^\pi_s   \\
 & + \Bigg\{ \int_{0+}^{t\wedge T_n}\mathrm{e}^{-qs}w'({U}^\pi_{s-})\mathrm{d}  [ X_s-(c-\int_0^1x\nu(\mathrm{d}x)) s-\sum_{0<u\leq s}\Delta X_u\mathbf{1}_{\{ |\Delta X_u| >1\}} ]   \Bigg\} \\
 & + \Bigg\{ \sum_{0<s\leq t\wedge T_n}\mathrm{e}^{-qs}[\Delta w({U}^\pi_{s-}+\Delta X_s)-w'({U}^\pi_{s-})\Delta X_s\mathbf{1}_{\{ |\Delta X_s| \leq1\}} ]\\
 & -   \int_{0+}^{t\wedge T_n}\int_{0+}^{\infty}\mathrm{e}^{-qs}\left[w({U}^\pi_{s-}-y)-w({U}^\pi_{s-})+w'({U}^\pi_{s-})y\mathbf{1}_{\{0<y\leq1\}}\right]\nu(\mathrm{d}y)\mathrm{d}s \Bigg\}.
\end{split}
\end{equation*}
By the L\'evy-It\^o decomposition the expression between the first pair of curly brackets is a zero-mean martingale and by the compensation formula (cf. \cite{kypbook}*{Corollary 4.6}) the expression between the second pair of curly brackets is also a zero-mean martingale. Hence we derive at
\begin{equation*}
\begin{split}
w(x) = &
 -\int_{0}^{t\wedge T_n}\mathrm{e}^{-qs}  \left[ (\Gamma-q)w({U}^\pi_{s-})-\ell^\pi(s)w'({U}^\pi_{s-})+ \ell^\pi(s) \right]  \mathrm{d}s \\
 & + \int_{0}^{t\wedge T_n}\mathrm{e}^{-qs}\ell^\pi(s)\mathrm{d}s + \mathrm{e}^{-q(t\wedge T_n)}w({U}^\pi_{t\wedge T_n}) + M_t,
\end{split}
\end{equation*}
where $\{M_t:t\geq0\}$ is a zero-mean $\mathbb{P}_x$-martingale.
Using  $w\geq 0$ and \eqref{HJB-inequality} leads to (since $0\leq \ell^\pi_s\leq\delta$)
\begin{equation*}
\begin{split}
w(x) \geq &
 \int_{0}^{t\wedge T_n}\mathrm{e}^{-qs}\ell^\pi(s)\mathrm{d}s  + M_t.
\end{split}
\end{equation*}

Now taking expectations and
letting $t$ and $n$ go to infinity and using the monotone convergence theorem we get, noting that $T_n\nearrow\sigma_0^\pi$ $\mathbb{P}_x$-a.s. and that $\pi\in\Pi_0$,
\begin{equation*}
w(x) \geq \mathbb{E}_x \left[ \int_{0}^{\sigma^\pi}\mathrm{e}^{-qs}\ell^\pi(s)\mathrm{d}s \right] =v_\pi(x).
\end{equation*}
Hence we proved $w(x)\geq v_*(x)$ for all $x>0$.

To finish the proof, note that $v_*$ is an increasing function (in the weak sence) and hence because $w$ is right-continuous at zero, $v_*(0)\leq\lim_{x\downarrow0}v_*(x)\leq\lim_{x\downarrow0}w(x)=w(0)$.


\end{doublespace}
\end{document}